\documentclass[11pt,letterpaper,reqno]{amsart}
\usepackage{amsfonts}
\usepackage{amsmath}
\usepackage{amsthm, amscd}
\usepackage{mathtools}
\usepackage{amssymb}
\usepackage{mathrsfs}
\usepackage{graphicx}
\usepackage{caption}
\usepackage{subcaption}
\usepackage{float}
\usepackage{xypic}
\usepackage[abs]{overpic}
\usepackage[alphabetic]{amsrefs}
\usepackage{etex}
\usepackage{tikz-cd}

\usepackage{hyperref}
\hypersetup{
	colorlinks,
	linkcolor=[rgb]{0,0,0.7},
	urlcolor=[rgb]{0,0,0.4},
	citecolor=[rgb]{0.4,0.1,0}
}

\allowdisplaybreaks

\theoremstyle{plain}\newtheorem{Theorem}{Theorem}[section]
\theoremstyle{plain}\newtheorem{Corollary}[Theorem]{Corollary}
\theoremstyle{plain}\newtheorem{Lemma}[Theorem]{Lemma}
\theoremstyle{plain}\newtheorem{Definition}[Theorem]{Definition}
\theoremstyle{plain}\newtheorem{Proposition}[Theorem]{Proposition}
\theoremstyle{plain}
\theoremstyle{plain}
\theoremstyle{plain}\newtheorem*{Claim*}{Claim}
\theoremstyle{plain}\newtheorem*{Theorem*}{Theorem}
\theoremstyle{plain}\newtheorem{Question}[Theorem]{Question}

\theoremstyle{remark}\newtheorem{remark}[Theorem]{Remark}
\theoremstyle{remark}
\theoremstyle{remark}\newtheorem*{Notation*}{Notation}

\theoremstyle{plain}
\makeatletter
\newtheorem*{rep@theorem}{\rep@title}
\newcommand{\newreptheorem}[2]{
\newenvironment{rep#1}[1]{
 \def\rep@title{#2 \ref{##1}}
 \begin{rep@theorem}}
 {\end{rep@theorem}}}
\makeatother
\newreptheorem{Theorem}{Theorem}
\newreptheorem{Proposition}{Proposition}
\newreptheorem{Corollary}{Corollary}

\numberwithin{equation}{section}

\usepackage{lipsum}

\DeclareMathOperator{\Emb}{Emb}
\DeclareMathOperator{\Diff}{Diff}

\title{Non-isotopic surfaces in $T^4\#(S^2\times S^2)$: an example}
\author{Jianfeng Lin}
\address{Yau Mathematical Sciences Center, Tsinghua University, Beijing 100084, China}
\email{linjian5477@mail.tsinghua.edu.cn}
\author{Yue Wu}
\address{Qiuzhen College, Tsinghua University, Beijing 100084, China}
\email{yue-wu23@mails.tsinghua.edu.cn}

\begin{document}

\maketitle

\begin{abstract}
We prove that there exist infinitely many embedded tori with a common geometric dual in $T^4\#(S^2\times S^2)$ that are homotopic, diffeomorphic, but not isotopic to each other, even after arbitrary many external stabilizations. These surfaces are obtained by applying the Norman trick to a fixed immersed surface, using non-homotopic tubing arcs. The isotopy classes of these surfaces are distinguished by homotopy classes of the 2-handles (relative to the boundary) in the complement of the image of the $0$- and $1$-handles. 
\end{abstract}

\section{Introduction}
\label{intro}

Given an oriented smooth 4-manifold $X$, there are two important equivalence relations for embedded surfaces in $X$: homotopy and smooth isotopy. While smooth isotopy implies homotopy, the converse is not always true, leading to a fundamental question in 4-dimensional  topology:
\begin{Question}\label{question}
    Suppose $\Sigma_1,\Sigma_2$ are embedded surfaces in $X$ that are homotopic to each other, what are the obstructions to smoothly isotoping $\Sigma_1$ to $\Sigma_2$?
\end{Question}

Question \ref{question} has been explored in many special cases, including  partial answers in both positive and  negative directions. Many results in this direction involve an operation called "stabilization". For example, after taking sufficiently many \emph{external stabilizations} (i.e. taking a connected sum of the ambient manifold with $S^2 \times S^2$ away from the surfaces), any two homologous embedded closed surfaces of the same genus with simply-connected complement are smoothly isotopic (see Remark \ref{surface stable isotopic}).
In \cite{Baykur_2015}, it is shown that any pair of homologous embedded surfaces $\Sigma_i, i = 1, 2$ in $X$ become smoothly isotopic after sufficiently many times of \emph{internal stabilizations} (i.e. attaching an embedded handle to the surface). Gabai's 4-dimensional lightbulb theorem \cite{Gabai4DLBT} states that if $\pi_1(X)$ has no $2$–torsion and $\Sigma_1,\Sigma_2$ are $G$–inessential, where $G$ is a common geometric dual sphere for $\Sigma_1$ and $\Sigma_2$, then homotopy implies isotopy. (See \cite{schwartz20214,Teichner2019,Kosanavic} for various generalizations.) Building on Gabai's result, Auckly-Kim-Melvin- Ruberman-Schwartz \cite{AKMRS} proved that any pair of homologous surfaces are isotopic after one external stabilization, provided that their complements are simply-connected and they are non-characteristic. In the other direction, for $X=\Sigma\times S^2$, \cite[Theorem 1.2]{LXZ25Dax} showed that there exist infinitely many embeddings of $\Sigma$ into $X$ which have a common geometric dual $G=\{*\}\times S^2$, are mutually homotopic, but non-isotopic to each other. In fact, the authors use the Dax invariant to classify the isotopy classes of embeddings of $\Sigma$ into $X$.

Consider two embeddings $i_1, i_2:\Sigma\rightarrow X$ that are homotopic to each other. 
Naturally, one could attempt to construct an isotopy between $i_1$ and $i_2$ inductively on handles. Fix a handle decomposition $H_0\cup H_1\cup H_2$ of $\Sigma$, where $H_0$ is the single 0-handle, $H_1$ is the union of $2g$ 1-handles and $H_2$ is the 2-handle.  For a dimensional reason, $i_2|_{H_0\cup H_1}$ is isotopic to $i_1|_{H_0\cup H_1}$. By the isotopy extension theorem, there exists an isotopy $\mathcal{I}$ from $i_2$ to some embeddeding $i_2':\Sigma\to X$ such that $i'_2|_{H_0\cup H_1}=i_1|_{H_0\cup H_1}$. By removing an open tubular neighborhood $\nu(i_1(H_0\cup H_1))$, we obtain a manifold $X'$ with boundary. Then the restrictions $i_1|_{H_2}$ and $i'_2|_{H_2}$ give properly embedded disks $D^2_1, D^2_2\hookrightarrow X^{\prime}$ that coincide with each other along their boundary.Therefore, $D^2_1\cup D^2_2:S^2\to X'$
represents an element $o(i_1,i_2,\mathcal{I})\in \pi_{2}(X')$. Consider the set 
\[
    S(i_1, i_2)=\{o(i_1,i_2,\mathcal{I})|\text{$\mathcal{I}$ is an isotopy  from $i_2$ to $i_2'$ with $i'_2|_{H_0\cup H_1}=i_1|_{H_0\cup H_1}$}\}.
\]
As a primary obstruction to this inductive approach, one examines whether $0$ lies in $S(i_1,i_2)$. If one can show that $0\notin S(i_1,i_2)$,equivalently, that $D^2_2$ is not homotopic to $D^2_1$ relative to the boundary after any isotopy $\mathcal{I}$, then $i_1$ is not isotopic to $i_2$. On the other hand, if $0\in S(i_1,i_2)$, then one can study a secondary obstruction, given by the relative Dax invariant $\operatorname{Dax}(D^2_1,D^2_2)$ (see \cite{gabai2021self,LXZ25Dax,Kosanavic}).) As a natural question, we may ask whether the primary obstruction can actually be nontrivial. This question can be stated more explicitly as follows.

\begin{Question}\label{quest: homotope 2-handle in the complement of 1-handle}
Given two embeddings $i_0,i_1:\Sigma\to X$ that are homotopic to each other, is it always possible to isotope $i_1$ such that the following two conditions are both satisfied ?
\begin{itemize}
    \item $i_1|_{H_0\cup H_1}=i_{0}|_{H_0\cup H_1}$. 
    \item The maps $i_{0}|_{H_2}, i_{1}|_{H_{2}}:D^2\to X\setminus i_0(\mathrm{int}(H_0\cup H_1))$ are homotopic relative to the boundary.
\end{itemize}
\end{Question}

In this paper, we give a negative answer to Question \ref{quest: homotope 2-handle in the complement of 1-handle}, even in the presence of a geometric dual.

\begin{Theorem}[Main theorem]\label{main thm}
     Let $X=T^4\#(S^2\times S^2)$. Then there is a collection of embeddings \[\{i_{\sigma}: T^2\rightarrow X\}_{\sigma\in \pi_{1}(T^2)}\] that satisfies the following conditions:
    \begin{enumerate}
        \item The surfaces $\{i_{\sigma}\}$ are homotopic, have diffeomorphic complements and have a common geometric dual $G$;
                \item For any  $\sigma_1\neq \sigma_2\in\pi_1(T^2)$, one has $0\notin S(i_{\sigma_1},i_{\sigma_2})$. In particular, $i_{\sigma_1}$ is not isotopic to $i_{\sigma_2}$;
        \item For any $\sigma_1\neq\sigma_2\in\pi_1(T^2)$, the embeddings $i_{\sigma_1}$ and $i_{\sigma_2}$ remain non-isotopic after arbitrary many external stabilizations.
    \end{enumerate}
\end{Theorem}

We now give an explicit construction of the embedded surfaces $\Sigma_{\sigma}=i_{\sigma}(T^2).$ First, we let $\Sigma_0=i_0(T^2)$ be obtained by taking the connected sum between $(\{*\}\times T^2,T^4)$ and $(\{*\}\times S^2,S^2\times S^2)$. 
Consider the sphere $G_1=S^2\times\{*\}\hookrightarrow X$, which is a geometric dual of $\Sigma_0$. Let $\{p_1\}=G\cap \Sigma_1$. Take a small disk $B_1\subset \Sigma_0$ such that $p_1\in \partial B_1$. Let $G_2$ be a parallel copy of $G_1$ that passes some point $p_{2}\in \partial B_1\setminus \{p_1\}$. For each $0\neq \sigma\in \pi_{1}(\Sigma_0,B_1)$, there is a unique (up to isotopy) embedded arc $\gamma_{\sigma}\hookrightarrow \Sigma_0\setminus \mathring{B_1}$ that goes from $p_1$ to $p_2$ and represents $\sigma$. Pick a standard arc $\kappa$ that connect $\Sigma_0$ and $G_1$ and lies in a small neighborhood of $p_1$ (see Section \ref{sec_constru} for the precise definition of $\kappa$) and consider the immersed surface $\Sigma'_{0}=\Sigma_0\#_{\kappa} G_1$ where $\#_{\kappa}$ means tubing along $\kappa$. Note that $\Sigma^{\prime}_0$ has a single double point $p_{1}$, we can apply the Norman trick to resolve this double point, using the arc $\gamma_{\sigma}$ and the geometric dual $G_2$. Namely, we attach a tube from $\Sigma'_0$ to $G_2$ along $\gamma_{\sigma}$. The resulting surface is the desired $\Sigma_{\sigma}$.

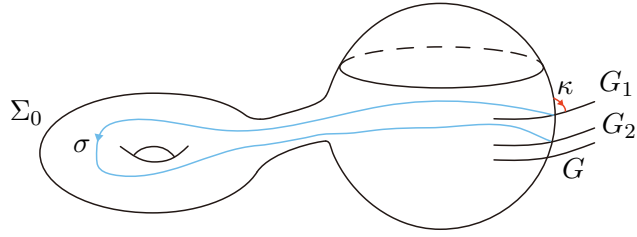
\begin{figure}
	\begin{overpic}[width=0.6\textwidth]{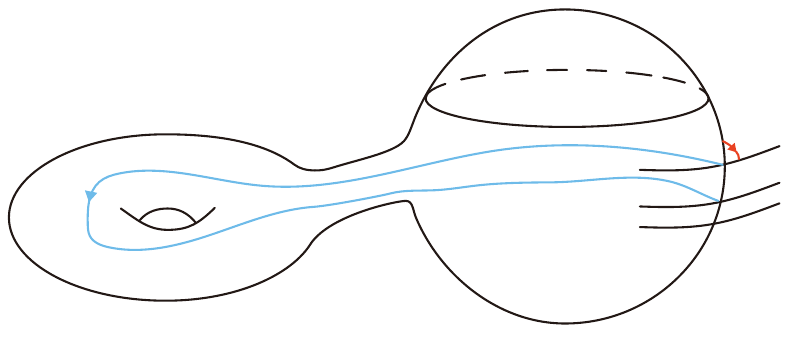}
        \put(-9,45){$\Sigma_0$}
		\put(200,23){$G$}
		\put(215,40){$G_{2}$}
		\put(214,56){$G_{1}$}
        \put(15,33){$\sigma$}
        \put(198,55){$\kappa$}
	\end{overpic}
	\caption{The construction of $\Sigma_{\sigma}$}
    \label{fig:Sigma_1}
\end{figure}

\begin{remark}
We conjecture that there is an alternative definition of the obstruction $S(i_{\sigma_1},i_{\sigma_2})$ in terms of the induced map $\bar{C}_{2}(T^2)\to \bar{C}_{2}(X)$ between the compactified configuration spaces. It would be interesting to give such an alternative definition for pairs of  embedded surfaces in a general 4-manifold $X$.    \end{remark}

 \begin{remark}\label{surface stable isotopic}
 By results of Boyer \cite[Theorem F]{Boyer93} and Galvin \cite[Theorem B]{Galvin} (see also \cite[Appendix A]{PowellOrsonGalvin} for an alternative proof), any pair of homologous surfaces with identical genus and simply-connected complements are isotopic after sufficiently many external stabilizations. One may ask whether the same result holds with a different condition that the complements are diffeomorphic. Theorem \ref{main thm} implies this is not true in general.
\end{remark}

\begin{remark}
Although we focus on the smooth category in this paper, we expect that analogous results can also be proved in the topological category. 
\end{remark}
Now we sketch the proof that $i_{\sigma_1}$ is not isotopic to $i_{\sigma_2}$ whenever $\sigma_1\neq \sigma_2$. First note that $i_{\sigma_1}$ and $i_{\sigma_2}$ are equal to each other on the complement of a disk $D^2\subset T^2$, so we may pick a handle decomposition $H_0\cup H_1\cup H_2$ of $T^2$ such that $i_{\sigma_1}|_{H_0\cup H_1}=i_{\sigma_2}|_{H_0\cup H_1}$. As a result, given any isotopy $\mathcal{I}$ from $i_{\sigma_1}$ to some embedding $i'_{\sigma_1}$ with $i'_{\sigma_1}|_{H_0\cup H_1}=i_{\sigma_2}|_{H_0\cup H_1}$, we have a loop in $\Emb(H_0\cup H_1,X)$. Here $\Emb(H_0\cup H_1,X)$ is the space of embeddings of $H_0\cup H_1$ into $X$. Let $\pi_2(X^{\prime},\partial D^2_0)$ be the homotopy class of maps $D^2\to X^{\prime}$ that agree with the embedded disk $D^2_0:=i_0(H_{2})\cap X'$ on the boundary. To prove Theorem \ref{main thm}, we study the action of $\pi_1(\Emb(H_0\cup H_1,X))$ on $\pi_2(X^{\prime},\partial D^2_0)$ via the isotopy extension theorem. It is shown that $\pi_1(\Emb(H_0\cup H_1,X))$ is generated by various spinning families. By applying the isotopy extension theorem to these generators, we obtain various barbell diffeomorphisms on $X'$. We  explicitly compute the maps on $\pi_2(X^{\prime},\partial D^2_0)$ induced by these barbell diffeomorphisms and conclude that $o(i_{\sigma_1},i_{\sigma_2},\mathcal{I})\neq 0$.
 
The paper is organized as follows. Section \ref{Preli} reviews some facts about spinning families of arcs and barbell diffeomorphisms. The construction of $i_{\sigma}$ is given in Section \ref{sec_constru}. The action of $\pi_1(\Emb(H_0\cup H_1,X))$ on $\pi_2(X^{\prime},\partial D^2_0)$ is calculated in Section \ref{sec_comput}. Theorem \ref{main thm} is proved in Section \ref{sec_proof}. 
\\
\\
\noindent
\textbf{Acknowledgments:} We would like to thank Boyu Zhang, Yi Xie and Mark Powell for inspiring discussions. J. Lin is partially partially supported by National Key R \& D Program of China (2025YFA1017500) and National Natural Science Foundation of China (12271281).

\section{Preliminary}
\label{Preli}
In this paper, we assume all manifolds, embeddings and isotopies are smooth.

\begin{Definition}
    Suppose $\Sigma$ is an embedded surface in a 4-manifold $M$, and $G$ is an embedded sphere in $M$. If $G$ has trivial normal bundle in $M$ and intersects $\Sigma$ transversely at exactly one point, then $G$ is called a \emph{geometric dual} of $\Sigma$.
\end{Definition}
\begin{Definition}
    Suppose $\Sigma_0,\Sigma_1$ are immersed surfaces in a 4-manifold $M$, and $\gamma:I\rightarrow M$ is an embedded path with $\gamma(i)\in \Sigma_i$, $i=0,1$. Assume $\gamma$ intersects $\Sigma_i$ transversely at $\gamma(i)$.  We pick a tubular neighborhood $\nu(\gamma)$ and a diffeomorphism $\rho:\nu(\gamma)\cong (-\epsilon,1+\epsilon)\times \mathbb{R}^3$ such that $\rho(\gamma)=[0,1]\times (0,0,0)$ and that $\rho(\nu(\gamma)\cap \Sigma_{i})=\{(i,x,y,0)\}$. Consider the embedded cylinder  $T(\gamma)=\rho(\{(t,x,y ,0)\mid t\in [0,1],x^2+y^2=1\})$ and the embedded disk $D_{i}:=\rho(\{(i,x,y ,0)\mid x^2+y^2\leq1\})$.
    After a canonical smoothing of  corners of the surface
    \[
        (\Sigma_0-\operatorname{int}D_0)\cup_{\partial D_0} T(\gamma)\cup_{\partial D^1} (\Sigma_1-\operatorname{int}D_1),
    \]
     we get a new immersed surface in $M$, called the surface obtained by tubing $\Sigma_0$ with $\Sigma_1$ along $\gamma$ and denoted by $\Sigma_0\#_{\gamma}\Sigma_1$.
    The choice of $\rho$ doesn't affect the isotopy class of $\Sigma_0\#_{\gamma}\Sigma_1$.
\end{Definition}

Let $\Sigma$ be an immersed surface with only transverse double points $x_1,\cdots, x_{n}$. Let $G$ be a geometric dual of $\Sigma$ and let $G_{1},\cdots, G_{n}$ be parallel copies. For each $1\leq i\leq n$, we take an embedded path $\gamma_{i}:I\to \Sigma$ such that $\gamma_{i}(0)=x_{i}$, $\{\gamma_i(1)\}=G_{i}\cap \Sigma$ and $\gamma_i(0,1)\subset \Sigma-\{x_{1},\cdots, x_{n}\}$. Then we take a tube $T(\gamma_{i})$ that connects $\Sigma$ with $G_{i}$. Even though the arcs $\gamma_{i}$ may intersect each other, we may set the radius of $\{T(\gamma_{i})\}$ to be all different so that they don't intersect. We let \[\Sigma'=\Sigma\#_{T(\gamma_1)}G_1\#_{T(\gamma_2)}G_2\#\cdots\#_{T(\gamma_n)}G_n.\]
Then $\Sigma'$ is an embedded surface. This procedure of removing double points is called \emph{Norman's trick}\cite{NORMAN1969}.



Let $M$ be a four manifold with nonempty boundary. Fix a family of disjoint neatly embedded arcs $\iota:\bigsqcup\limits_{1\leq u\leq m} I_u\rightarrow M$ in $M$, equipped with a fixed normal section $s_0$. (A normal section on a submanifold means a nowhere vanishing section of its normal bundle. ) Here, $I_u$ is a copy of $I=[0,1]$. By abuse of notation, we also denote the image of $I_u$ under $\iota$ by $I_u$.
\begin{Definition}
     Define $\Emb_{\partial}(\sqcup_u I_u,M)$ to be the space of family of neatly embedding arcs in $M$ which coincide with $\iota$ near boundary. We take $\iota$ to be the basepoint of $\Emb_{\partial}(\sqcup_u I_u,M)$. Define $\Emb^{\prime}_{\partial}(\sqcup_u I_u,M)$ to be the space of embeddings $i:\sqcup_u I_u\rightarrow M$ equipped with a normal section $s$ which coincide with $(\iota,s_0)$ near boundary. And we take  $(\iota,s_0)$ to be the basepoint of $\Emb^{\prime}_{\partial}(\sqcup_u I_u,M)$.
\end{Definition}
\begin{Definition}
    For subsets $A,B$ of $M$, we denote by $\pi_1(M;A,B)$ the set of homotopy classes of paths $\gamma:[0,1]\rightarrow M$ with $\gamma(0)\in A,\gamma(1)\in B$ and $\operatorname{int} \gamma\cap(A\cup B)=\emptyset$.
\end{Definition}
 \begin{Definition}\label{defi:spinning arc around sphere}
    Let $S$ be an embedded sphere in $M$ with trivial normal bundle and disjoint from $\iota$ and $s$ is a fixed normal section on $S$. Given $1\leq j\leq m$  and a path $\lambda\in\pi_1(M;I_j,S)$, we may assume $\operatorname{int}\lambda\cap(\iota\cup S)=\emptyset$. For $u\neq j$, keep $(I_u,s_0|_{I_u})$ fixed. Homotope $I_j$ along $\lambda$ to get an arc intersects $S$ in a segment $\overline{I}_j$ with normal section $s_{\overline{I}_j}$, swipe $\overline{I}_j$ along $S$ and the normal section is always given by the restriction of $s$. Finally homotope $\overline{I}_j$ back along $\lambda^{-1}$. Then we get a loop in $\Emb_{\partial}^{\prime}(\sqcup_u I_u,M)$, denoted by $\hat{\eta}_{\lambda}\in\pi_1(\Emb_{\partial}^{\prime}(\sqcup_u I_u,M))$. This loop is called the \emph{spinning families of arcs obtained by spinning $I_j$ around $S$ along $\lambda$}. 
 \end{Definition}
 
 \begin{Definition}\label{defi: spinning family}
    Given $1\leq u,v\leq m$ and $g\in\pi_1(M;I_u,I_v)$, represent $g$ by an embedded path $\alpha$ from $I_u$ to $I_v$, such that $\operatorname{int}\alpha\cap \iota=\emptyset$. Let $m_v$ be a meridian of $I_v$, and we homotope $\alpha$ slightly to get a path $\alpha^{\prime}$ from $\iota$ to $m_v$. We further assume $\operatorname{int}\alpha^{\prime}\cap(\iota\cup m_v)=\emptyset$. Then we get spinning families $\hat{\tau}_g\in\pi_1(\Emb_{\partial}^{\prime}(\sqcup_u I_u,M))$ obtained by spinning $I_u$ around $m_v$ along $g$ as described in Definition \ref{defi:spinning arc around sphere}.
 \end{Definition}
In some cases, we need to consider a spinning families of arcs without normal section. Note that we have a fibration:
\begin{equation}
\label{fibration of normal section}
    \Emb_{\partial}^{\prime}(\sqcup_u I_u,M)\rightarrow \Emb_{\partial}(\sqcup_u I_u,M))
\end{equation}
We use $\tau_g, \eta_{\lambda}\in \pi_1(\Emb_{\partial}(\sqcup_u I_u,M))$ to denote the image of \[\hat{\tau}_g,\hat{\eta}_{\lambda}\in\pi_1(\Emb_{\partial}^{\prime}(\sqcup_u I_u,M))\] under the induced map by (\ref{fibration of normal section}).
\begin{Definition}\label{spinning of framing}
    Given $1\leq u\leq m$, fix a trivialization of normal bundle of $I_u$ which extends $s_0|{I_u}$. Then any normal section on $I_u$ which coincides with $s_0|_{I_u}$ near boundary gives a point in $\Omega S^2$. Pick a loop of normal sections $\{s_u^{\theta}\}_{0\leq\theta\leq1}$ on $I_u$ that represents a generator of $\pi_1(\Omega S^2)\cong\mathbb{Z}$, then define $\xi_u\in\pi_1(\Emb_{\partial}^{\prime}(\sqcup_u I_u,M))$ as 
    \[
        (\iota, \sqcup_{v\neq u} s_0|_{I_v}\sqcup s^{\theta}_{u}),\quad 0\leq\theta\leq 1.
    \]
\end{Definition}

Note that the fundamental group of the fiber of (\ref{fibration of normal section}) is generated by $\{\xi_u\}_{1\leq u\leq m}$. Hence, to investigate the fundamental group $\pi_1(\Emb_{\partial}^{\prime}(\sqcup_u I_u,M))$, it remains to explore generators of $\pi_1(\Emb_{\partial}(\sqcup_u I_u,M))$. 
\begin{Definition}
    Consider the map $\mathcal{F}:\pi_1(\Emb_{\partial}(\sqcup_u I_u,M))\rightarrow \prod\limits_u\pi_2(M)$ by viewing a loop $\alpha_t:I\rightarrow \Emb_{\partial}(\sqcup_u I_u, M)$ as a map $\alpha:I\times \sqcup_u I_u\rightarrow M$, which lies in $\prod\limits_u\pi_2(M)$. The \emph{Dax group} $\pi_1^D(\Emb_{\partial}(\sqcup_u I_u,M))$ is defined to be the kernel of $\mathcal{F}$.
\end{Definition}
Therefore, there is an exact sequence:
\begin{equation}\label{split as three barbell}
    \pi_1^D(\Emb_{\partial}(\sqcup_u I_u,M))\hookrightarrow \pi_1(\Emb_{\partial}(\sqcup_u I_u,M))\xrightarrow{\mathcal{F}} \prod\limits_u\pi_2(M).
\end{equation}
By \cite[Theorem 0.3]{gabai2021self}, the Dax group $\pi_1^D(\Emb_{\partial}(\sqcup_u I_u,M))$ is generated by $\{\tau_g\mid g\in\pi_1(X;I_u,I_v)\}_{1\leq u,v\leq m}$, where $\tau_g$ is the image of the spinning family of arcs $\hat{\tau}_g$ in $\pi_1(\Emb_{\partial}(\sqcup_u I_u,M))$. In some special cases, $\prod\limits_u\pi_2(M)$ can be generated by the image of some spinning families of arcs in $\pi_1(\Emb_{\partial}(\sqcup_u I_u,M))$ under $\mathcal{F}$ and then we can describe a set of generators of $\pi_1(\Emb_{\partial}(\sqcup_u I_u,M))$ by spinning families of arcs.

\subsection{Barbell diffeomorphism}Now we briefly review the barbell diffeomorphisms, introduced by Budney-Gabai\cite{BG2019}.

Let $S_0\subset\mathbb{R}^3$ be the sphere centered at $(-1,0,0)$ with radius $\frac{1}{2}$, $S_1\subset\mathbb{R}^3$ be the sphere centered at $(1,0,0)$ with radius $\frac{1}{2}$ and $\gamma$ is the straight line from $(-\frac{1}{2},0,0)$ to $(\frac{1}{2},0,0)$. A \emph{model barbell} $\mathcal{B}\subset\mathbb{R}^4$ is obtained by taking products of  $\nu(S_0\cup\gamma\cup S_1)$ with $[-1,1]$, where $\nu(S_0\cup\gamma\cup S_1)$ is a $\epsilon-$regular neighborhood of $S_0\cup\gamma\cup S_1$ in $\mathbb{R}^3$.  We denote the $t-$slice of $\mathcal{B}$ by $\mathcal{B}_t$, i.e. $\mathcal{B}_t=\nu(S_0\cup\gamma\cup S_1)\times t$. A barbell in a four manifold $M$ is an embedding of $\mathcal{B}$ in $M$. 

Let $E_0$ be the disk $(-1,0)\times D^2_{\epsilon}(\frac{1}{2},0)$ and $E_1$ be the disk $(1,0)\times D^2_{\epsilon}(\frac{1}{2},0)$. Let $B_0,B_1$ be the complementary 4-balls of $\mathcal{B}$ in $\mathbb{R}^4$. To define the barbell diffeomorphism, we first perturb $E_0$ in the $y-$direction to get a copy $F_0$. Then both $F_0\cap\mathcal{B}_{-1}$ and $F_0\cap\mathcal{B}_{1}$ are two arcs. Spin $F_0\cap\mathcal{B}_{-1}$ around $S_1\times t$ along $\gamma$, and finally get the arc $F_0\cap\mathcal{B}_{1}$. Then we have a family of embedded arcs $F_1:[-1,1]\rightarrow \Emb(I,\mathcal{B}\cup B_0\cup B_1)$, which can be also viewed as an embedded disk in $\mathcal{B}\cup B_0\cup B_1$. The loop $\lambda$ in $\Emb(B_0,\mathcal{B}\cup B_0\cup B_1)$ is obtained by first sweeping across $F_1$ and then sweeping back across $F_0$.
Applying isotopy extension to the loop $\lambda$, we get a diffeomorphism $\beta_{(S_0,\gamma,S_1)}:\mathcal{B}\rightarrow\mathcal{B}$ fixing $\partial \mathcal{B}$, which is called the \emph{barbell diffeomorphism} corresponding to $S_0\cup\gamma\cup S_1$. If $M$ is a four manifold and $\mathcal{B}$ is a barbell in $M$, then the barbell diffeomorphism is obtained by extending $\beta_{(S_0,\gamma,S_1)}$ using identity.

Let $M$ be a four manifold and let $\iota:\bigsqcup\limits_{1\leq u\leq m} I_u\rightarrow M$  be a neatly embedding. By the isotopy extension theorem,we have a canonical map
\[
    e:\pi_1(\Emb_{\partial}(\sqcup_u I_u,M))\rightarrow \pi_0(\Diff_{\partial}(M-\nu(\iota)))
\]
There is a close relation between barbell diffeomorphisms and spinning families of arcs. The following proposition directly follows from the definition.
\begin{Proposition}
     Let $S$ be an embedded sphere in $M$ with trivial normal bundle and disjoint from $\iota$ and $m_u$ be a fixed meridian sphere of $I_u$. Then
     \begin{enumerate}
         \item For any $\lambda\in\pi_1(M;I_j,S)$, $e(\eta_{\lambda})=\beta_{(m_j,\lambda,S)}$;
         \item For any $g\in\pi_1(M;I_u,I_v)$, $e(\tau_{g})=\beta_{(m_u,g,m_v)}$. $\makeatletter\displaymath@qed\makeatother$
     \end{enumerate}
\end{Proposition}

\section{Construction of the surfaces}
\label{sec_constru}

Fix a handle decomposition $H_0\cup H_1\cup H_2$ on $T^2$, where $H_0$ is the single 0-handle, $H_1=H_1^1\cup H_1^2$ is the union of two 1-handles and $H_2$ is the 2-handle. Let $e^0\cup e^1_1\cup e_2^1\cup e^2$ be the corresponding CW structure on $T^2$. We pick the base point $x_0=e_0\subset H_0$.

Take a point $b_1\in \mathring{H}_2\subset T^2$ and $b_2\in S^2$. Take a local chart $\varphi_{1}: U_{1}\xrightarrow{\cong} \mathbb{R}^4$ near $(b_{1},b_{1})\in T^4$ and a local chart $\varphi_{2}: U_{2}\xrightarrow{\cong} \mathbb{R}^4$ near $(b_{2},b_{2})\in S^2\times S^2$. We require that 
\[
\varphi_{1}((\{b_1\}\times T^2)\cap U_{1})=\{(0,0,z,w)\}\] and that \[\varphi_{2}((\{b_2\}\times S^2)\cap U_{2})=\{(0,0,z,w)\}.\] Let $D\in \mathbb{R}^4$ be the ball with center $(2,0,0,0)$ and radius $1$.  For $i=1,2$, let $D^4_{i}=\varphi^{-1}_{i}(D)\subset U_{i}$. Let $\gamma_{i}:I\to U_{i}$ be defined by $\gamma_{i}(t)=\varphi^{-1}_{i}(t,0,0,0)$. Now we form the connected sum 
\[
X=(T^4\setminus \mathring{D}^4_{1})\sqcup ((S^2\times S^2)\setminus \mathring{D}^4_{2})/\sim \]
where$\sim$ is generated by \[\varphi^{-1}_{1}(x,y,z,w)\sim \varphi^{-1}_{2}(x,y,z,-w),\quad \forall (x,y,z,w)\in \partial D.
\]

Let $\Sigma^{T}_0=(\{b_1\}\times T^2)-\mathring{D}_{1}^4$ and $\Sigma^S_0=(\{b_2\}\times S^2)-\mathring{D}_{2}^4$. Then 
\[
\Sigma_0=\Sigma^T_0\cup \Sigma^S_0. 
\]
is an embedded surface in $X$. 
We fix a parameterization $i_0: T^2\xrightarrow{ \cong}\Sigma_0\hookrightarrow X$ such that $i_0(H_0)=\{b_{1}\}\times H_0$, $i_0(H_{1})=\{b_{1}\}\times H_{1}$ and $\Sigma^S_0\subset i_0(H_2)$. 

We fix a 2-dimensional disk $B_1\subset \Sigma^S_0$. Then for any $p\in B_1$, there is a geometric dual $G_{p}=S^2\times \{*\}\subset X$ of $\Sigma_0$ that passes $p$. Pick two different points $p_{1},p_{2}\in \partial B_1$ and let $G_{i}=G_{p_{i}}$. We take a local chart $\varphi:U\xrightarrow{\cong} \mathbb{R}^{4}$ of $X$ near $p_{1}$ such that $\varphi(U\cap \Sigma_0)=\mathbb{R}^2\times \{0\}\times \{0\}$ and $\varphi(U\cap G_{1})=\mathbb{R}^2\times \{0\}\times \{0\}$. Let $\kappa=\varphi^{-1}\{(t,0,1-t,0)\mid t\in [0,1]\}$. Then $\kappa$ is an arc connecting $\Sigma_0$ and $G_{1}$. By tubing along $\kappa$, we obtain an immersed surface $\Sigma'_{0}:=\Sigma_{0}\#_{\gamma}
G_{1}\looparrowright X$ with a single double point $p_1$. Now we use the Norman's trick to remove this double point. For this, we pick any element  $\sigma\in \pi_{1}(\Sigma,B_1)\cong \pi_{1}(\Sigma)$, represented by an embedded arc $\sigma\hookrightarrow  \Sigma-\mathring{B_1}$ that goes from $p_1$ to $p_{2}$. Then we let \[\Sigma_{\sigma}:=\Sigma'_{0}\#_{\sigma}\overline{G_{2}}=\Sigma_0\#_{\kappa}G_{1}\#_{\sigma}\overline{G_{2}}.\]

For our later discussion, we fix a point  $p\in \textrm{int} B_1$ and let $G=G_{p}$. We also fix a point $b'_{2}\neq b_{2}\in S^2$ and let $R=\{b'_{2}\}\times S^2$. Then $G,R$ are embedded spheres in $X$.

\section{Action of $\pi_1(\Emb(H_0\cup H_1,X))$ on $\pi_2(X^{\prime},\partial D^2_0)$}
\label{sec_comput}
Recall from Section \ref{sec_constru} that $T^2=H_0\cup H_1\cup H_2$ is a fixed handle decomposition on $T^2$, and $i_0:T^2\rightarrow X$ is an embedding with image $\Sigma_0$. Consider the manifolds with boundary
\[
X_{1}=X-\nu(i_0(H_0))
\]
and
\[X'=X-\nu(i_0(H_0\cup H_{1})).
\]
Let $D^2_0, D^2_{\sigma}$ be the intersections $\Sigma_0\cap X^{\prime}, \Sigma_{\sigma}\cap X^{\prime}$, respectively. Then $D^2_0, D^2_{\sigma}$ are properly embedded disks in $X^{\prime}$.

Denote by $\pi_2(X^{\prime},\partial D^2_0)$ the set of homotopy classes of maps $D^2\rightarrow X^{\prime}$ that coincide with $D^2_0$ in a neighborhood $U_0$ of $\partial D^2_0$. Note that $D^2_0\cup G\cup R$ is simply connected. So we have a canonical identification $\pi_1(X')=\pi_{1}(X', D^2_0\cup G\cup R)$. Given an element $[D^2_{1}]\in \pi_2(X^{\prime},\partial D^2_0)$, we have the equivariant intersection numbers 
\[D^2_{1}\cdot D^2_{0},\  D^2_{1}\cdot G,\ D^2_{1}\cdot R\in \mathbb{Z}[\pi_1(X')].\]
Consider the ideal 
  \[
    \mathcal{J}:=\{\sum^{k}_{i=1} a_{i}g_{i}\mid a_{i}\in \mathbb{Z},\ g_{i}\in \pi_{1}(X'),\ \sum^k_{i=1}a_{i}=0\}\subset \mathbb{Z}[\pi_1(X^{\prime})].
    \]    
and the group 
\[
\widetilde{\mathcal{J}}:=\mathcal{J}\oplus(\mathbb{Z}[\pi_1(X^{\prime})])^{\oplus2}.
\]
\begin{Lemma} For any $[D^2_1]\in  \pi_2(X^{\prime},\partial D^2_0)$, we have $D^2_{1}\cdot D^2_{0}\in \mathcal{J}$. Furthermore, the map   
    \begin{equation}\label{eq: pi_2 X'}
    \varphi:\pi_2(X^{\prime},\partial D^2_0)\to  \widetilde{\mathcal{J}},       \end{equation} 
defined by 
\[
\varphi([D^2_{1}])= (D^2_1\cdot D^2_{0}, D^2_{1}\cdot G, D^2_{1}\cdot R )
\]
is a well-defined isomorphism.
\end{Lemma}
\begin{proof} Given $[D^{2}_{1}]\in \pi_{2}(X',\partial D^2_0)$, we can form a map 
\[
S^{2}=D^2_{1}\cup D^2_0\to X'. 
\]
This gives a bijection 
     \[
        \pi_2(X^{\prime},\partial D^2_0)\cong \pi_2(X^{\prime},x^{\prime})\cong \pi_2(\widetilde{X}^{\prime},\widetilde{x}')\cong H_2(\widetilde{X}^{\prime}).
     \]
     
     Here $x'\in \partial D^2_0$,  $\widetilde{X^{\prime}}$ is the universal cover of $X^{\prime}$ and $\widetilde{x}'$ is a preimage of $x'$. It suffices to show that the map 
     \[
\pi_{2}(\widetilde{X}^{\prime},\widetilde{x}')\to \widetilde{\mathcal{J}}
     \]
     defined by taking intersection numbers with $D^2_0, G, R$ is a bijectiton. 
   
   Recall that $X'=X-\nu(i_0(H_0\cup H_{1}))$, then $\widetilde{X^{\prime}}=(\mathbb{R}^4-W\times\mathbb{Z}^2)\#_{\mathbb{Z}^4}(S^2\times S^2)$, where  \[W
  =      \{(x,y)\in\mathbb{R}^2\mid x\in\mathbb{Z} \text{ or }y\in\mathbb{Z}\}.
     \]
     Then $\pi_{2}(\widetilde{X}^{\prime},\widetilde{x}')\textbf{}\cong H_2(\widetilde{X}')$ is a free abelian group generated by meridians of $W$, lifts of $G$ and lifts of $R$. 
     A computation of equivariant intersection numbers of these generators with $D^2_0, G, R$ finishes the proof.
\end{proof}

For any $[D^2_{1}], [D^2_{2}]\in \pi_2(X^{\prime},\partial D^2_0)$ and $[D^2]$, we use $[D^2_{1}]-[D^2_2]$ to denote the difference 
\[
\varphi([D^2_{1}])-\varphi([D^2_2])\in \widetilde{\mathcal{J}}.
\]
Under the isomorphism $\widetilde{\mathcal{J}}\cong \pi_{2}(X')$, the element $[D^2_{1}]-[D^2_2]$ is the homotopy class of the map 
\[
S^2=D^2_{1}\cup D^2_2\to X'.
\] 
For example, we have $[D^2_{\sigma}]-[D^2_0]=(0,0,1-\sigma)$ for any $\sigma\in\pi_1(\Sigma_0)$.

\begin{Definition}
   Suppose  $i_1, i_2:T^2\rightarrow X$ are embeddings that coincide with $i_0$ on $H_0\cup H_1$.  Let $\mathcal{I}$ be any isotopy $\mathcal{I}$ from $i_2$ to some embedding $i_2'$ such that $i'_2|_{H_0\cup H_1}=i_1|_{H_0\cup H_1}$. Then $i_{1}(T^2)\cap X'$ and $i'_{2}(T^2)\cap X'$ are properly embedded disks in $X^{\prime}$ that coincide with $D^2_0$ near the boundary. We denote them by $D^2_1, D^2_2$. And we denote $[D^2_1]-[D^2_2]\in\widetilde{\mathcal{J}}$ by $o(i_1,i_2,\mathcal{I})$.
\end{Definition}

Now we consider embedded spheres $R,G, m_{1}, m_{2}:S^2\hookrightarrow X'$. Here $R,G$ are defined at the end of Section \ref{sec_constru}, and $m_{u}$ is a meridian of the arc $I_{u}\hookrightarrow X_1$. We use $x_{G}$ to denote the intersection point between $D^2_{\sigma}$ and $G$. And we use $x^{\pm}_{u}$ to denote the intersection point between $D^{2}_{\sigma}$ and $m_u$. We may assume that $x_{G}, x^{\pm}_{u}$ are independent with $\sigma$ and are contained in a small neighborhood $U_0$ of $\partial D^2_0$.  Here $x^{+}_{u}$ (resp. $x^{-}_{u}$) is obtained by pushing a point $x_{u}\in I_u$ slightly in the positive (resp. negative) direction. We orient $m_{u}$ such that $x^{+}_{u}$ is a positive intersection point and $x^{-}_{u}$ is a negative intersection point.

Tubing $D^2_{\sigma}$ with one of these spheres will change its homotopy class in $\pi_{2}(X',\partial D^2_0)$. The following lemma describes this change. The proof is a straightforward calculation of equivariant intersection numbers. 

\begin{Lemma}\label{lem: tubing} (1) Let $\gamma:I\to X'$ be a path from $U_0$ to $m_{u}$. We compose  $\gamma$ with paths in $m_{u}$ to get paths $\gamma^{\pm}$ from $U_0$ to $x^{\pm}_{u}$. Then for any $[D^2_{1}]\in \pi_{1}(X)$, we have 
\[
[D^2_{1}\#_{\gamma} m_{u}]-[D^2_{1}]=(\gamma^+-\gamma^-,0,0)\in \widetilde{\mathcal{J}}.
\]
Here we regard $\gamma^{\pm}$ as  elements in $\pi_{1}(X', D^2_0)\cong \pi_{1}(X')$.

(2) Let $\gamma:I\to X'$ be a path from $U_0$ to $R$. We compose  $\gamma$ with a path in $G\cup R$ to  a path $\gamma'$ from $U_0$ to itself. Then for 
any $[D^2_1]$, we have 
\[
[D^2_{1}\#_{\gamma} R]-[D^2_{1}]=(0,\gamma',0)\in \widetilde{\mathcal{J}}.
\]
Here we regard $\gamma'$ as an element in $\pi_{1}(X', D^2_0)\cong \pi_{1}(X')$.

(3) Let $\gamma:I\to X'$ be a path from $U_0$ to $G$.
Then for 
any $[D^{2}_{1}]\in \pi_{2}(X',\partial D^2_0)$, we have 
\[
[D^2_{1}\#_{\gamma} G]-[D^2_{1}]=(0,0,\gamma)\in \widetilde{\mathcal{J}}.
\]
Here we regard $\gamma$ as an element in $\pi_{1}(X',D^2_0\cup G)\cong \pi_{1}(X')$.
\end{Lemma}

\subsection{Fundamental group of embedding spaces}
\begin{Definition}
   Define $\Emb(x_0,X)$, $\Emb(H_0,X)$,  and $\Emb(H_0\cup H_1, X)$ to be the space of embeddings from $x_0$, $H_0$ and $H_0\cup H_1$ into $X$ respectively. We pick base points to be the restriction of $i_0$.
\end{Definition}
By \cite{Palais1960},  there is a fibration tower of embedding spaces :
\[
    \Emb(H_0\cup H_1, X)\xrightarrow{r_1} \Emb(H_0,X)\xrightarrow{r_0} \Emb(x_0,X),
\]
where $r_j$ is given by restriction and $\Emb(x_0,X)$ is the space of maps from $x_0$ to $X$. Obviously, $\Emb(x_0,X)$ is exactly $X$ itself. Let $F_i$ be the preimage of the basepoints under $r_i$. Namely,
\[
F_0=\{\text{smooth embedding } f:H_0\rightarrow X\mid f(x_0)=i_0(x_0) \},
\]
\[
 F_1=\{\text{smooth embedding } f:H_0\cup H_1\rightarrow X\mid f|_{H_0}=i_0|_{H_0}\}.
\]

Recall that $X_{1}=X-\nu(i_0(H_0))$. For $u=1,2$, denote $e^{1}_{u}\cap X_1$ by $I_{u}$. Recall that $\Emb^{\prime}_{\partial}(I\sqcup I,X_1)$ is the space of neat embeddings $i:I\sqcup I\rightarrow X_1$ equipped with a normal section $s$ which coincide with $i_0|_{I_1\sqcup I_2}$ and its normal section $s_0$ near boundary. Here $s_0$ is the image under $i_0$ of a fixed normal section of $I_1\sqcup I_2$ in $T^2$ that is tangent to $\partial i_0^{-1}(X_1)$. 
\begin{Lemma}
     The homotopy types of $F_j$ are described as follows:
     \begin{enumerate}
         \item $F_0$ is homotopy equivalent to the Stiefel manifold $V_2(T_{(x_0,b)}X)$, thus is simply connected;
         \item $F_1$ is homotopy equivalent to $\Emb^{\prime}_{\partial}(I\sqcup I,X_1)$
     \end{enumerate}
 \end{Lemma}
\begin{proof}
    The proof is straightforward.
\end{proof}
Since $F_0$ is simply-connected, we have $\pi_1(\Emb(H_0,X))\cong \pi_1(\Emb(x_0,X))$, which is just $\pi_1(X,i_0(x_0))$. Therefore we have the following exact sequence:
\begin{equation}\label{split as basepoint}
    \pi_1(F_1)\rightarrow \pi_1(\Emb(H_0\cup H_1, X))\xrightarrow{\mathrm{Res}} \pi_1(X,i_0(x_0))
\end{equation}
    
\begin{Definition}
    
Take any element in $\pi_1(\Emb(H_0\cup H_1, X))$, represented by a loop $\{f_t:H_0\cup H_1\rightarrow X\}_{t\in[0,1] }$. We restrict $f_{t}$ to $(H_0\cup H_{1})\cap i_0^{-1}(X-\mathring{X}')$ and apply the isotopy extension theorem to get a path of embeddings $\{f'_{t}: X-\mathring{X}'\hookrightarrow X\}$. Then we apply the isotopy extension theorem again to get a path $\{\widetilde{f}_{t}: X\to X\}$ of diffeomorphisms. Define the map 
\begin{equation}\label{eq: extension}
e:\pi_1(\Emb(H_0\cup H_1, X))\to \pi_0(\Diff(X',\partial D^2_0)) 
\end{equation}
by $e([\{f_{t}\}])= [\{\widetilde{f}_{t}|_{X'}\}]$. 
\end{Definition}
Via $e$, we get an action of $\pi_1(\Emb(H_0\cup H_1, X))$ on $\pi_2(X^{\prime},\partial D^2_0)$.

Recall that $R=\{b'_{2}\}\times S^2\hookrightarrow X$ is a sphere disjoint from $\Sigma_0$.

\begin{Proposition}\label{move of basepoint}
    For any $\theta\in \pi_1(\Emb(H_0\cup H_1, X))$, there exists $\theta^{\prime}\in\pi_1(\Emb(H_0\cup H_1, X))$, such that the following conditions hold:
    \begin{itemize}
        \item $\mathrm{Res}(\theta^{\prime})=\mathrm{Res}(\theta)$,
        \item $e(\theta^{\prime}) ([D^2_0])=[D^2_0\#_{\overline{\theta}}R]$, where $\overline{\theta}=\mathrm{Res}(\theta)$.
    \end{itemize}
       In particular,  $e(\theta^{\prime})([D^2_0])-[D^2_0]=(0,1-\overline{\theta},0)\in \widetilde{\mathcal{J}}$
\end{Proposition}
\begin{proof}
By our construction of $\Sigma_0$, we have $\Sigma_0=(\{b_1\}\times T^2)\#_{\gamma_0\cdot\gamma_1}(\{b_2\}\times S^2)$, where $\gamma_0$ is a trivial path contained in $T^4-\mathrm{int}(D_1^4)$ and $\gamma_1$ is a trivial path contained in $S^2\times S^2-\mathrm{int}(D_2^4)$. (Here a trivial path means a straight line under a standard local coordinate.)

Pick a loop in $T^4$ representing $\overline{\theta}\in\pi_1(X)\cong\pi_1(T^4)$, still denoted by $\overline{\theta}$. Consider the loop $\overline{\beta}:I\rightarrow \Emb(T^2, T^4)$ of embedded surfaces in $T^4$:
\[
    \overline{\beta}_t(x)=\overline{\theta}\cdot i_{b_1}(x),
\]
where $i_{b_1}$ is the inclusion $T^2\rightarrow T^4, x\rightarrow (b_1,x)$, and $\overline{\theta}\cdot$ represents the multiplication with $\overline{\theta}$ using the group structure on $T^4$. By dimensional reason, we can perturb $\overline{\beta}$ such that the image of $\overline{\beta}_t$ is contained in $T^4-\mathrm{int}(D_1^4)$ for any $t\in [0,1]$.  By the isotopy extension, we can lift $\overline{\beta}\in \pi_1(\Emb(T^2,T^4-\mathrm{int}(D_1^4)))$ to a loop $\beta$ in $\Diff_{\partial}(T^4-\operatorname{int}(D_1^4))$. Note that $\beta_t\circ\gamma_0(1)=\gamma_1(0)$ for any $t$, thus we can form the family of tubed surfaces:
\[
    \Sigma_t:=\beta_t\circ i_{b_1}(T^2)\#_{(\beta_t\circ\gamma_0)\cdot \gamma_1}(\{b_2\}\times S^2)
\]
which is an isotopy of $i_0:T^2\rightarrow X$. Restricting this isotopy to $H_0\cup H_1$, we get the required loop $\theta^{\prime}\in\pi_1(\Emb(H_0\cup H_1, X))$.
\end{proof}

To investigate elements in $\pi_1(\Emb(H_0\cup H_1, X))$, we need to find generators of $\pi_1(\Emb^{\prime}_{\partial}(I\sqcup I,X_1))$ by the exact sequence \eqref{split as basepoint}. Recall from Section \ref{Preli} that there is a fibration of embedding spaces:
\[
    \Emb^{\prime}_{\partial}(I\sqcup I,X_1)\rightarrow \Emb_{\partial}(I\sqcup I,X_1),
\]
and the fundamental group of the fiber is generated by $\xi_1,\xi_2$ defined in  Definition \ref{spinning of framing}. Given elements $\lambda\in\pi_1(X_1;I_v,G), g\in\pi_1(X_1;I_u,I_v)$, we have  spinning families of arcs $\hat{\eta}_{\lambda},\hat{\tau}_g\in\pi_1(\Emb_{\partial}^{\prime}(I\sqcup I,X_1)).$ (See Definition \ref{defi: spinning family} and Definition \ref{spinning of framing}.) Their images in $\pi_1(\Emb_{\partial}(I\sqcup I,X_1))$ are denoted by $\eta_{\lambda}, \tau_g$ respectively. Given any path $\gamma\in\pi_1(X_1;I_v,R)$, we can also spin $I_j$ around $R$ along $\gamma$ to get a spinning family of arcs $\hat{\rho}_{\gamma}\in\pi_1(\Emb_{\partial}(I\sqcup I,X_1))$. And the image of $\hat{\rho}_{\gamma}$ in $\pi_1(\Emb_{\partial}(I\sqcup I,X_1))$ is denoted by $\rho_{\gamma}$.
\begin{Proposition}\label{description of normal section}
   The fundamental groups of $\Emb_{\partial}(I\sqcup I,X_1)$ and $\Emb^{\prime}_{\partial}(I\sqcup I,X_1)$ have the following descriptions:
   \begin{enumerate}
       \item $\pi_1(\Emb_{\partial}(I\sqcup I,X_1))$ is an abelian group generated by elements of the form $\tau_g, \eta_{\lambda},\rho_{\gamma}$;
       \item $\pi_1(\Emb_{\partial}^{\prime}(I\sqcup I,X_1)))$ is generated by elements of the form $\hat{\tau}_g, \hat{\eta}_{\lambda}, \hat{\rho}_{\gamma}$ and $\{\xi_u\}_{u=1,2}$. Moreover, $\pi_1(\Emb_{\partial}^{\prime}(I\sqcup I,X_1)))$ is abelian. 
   \end{enumerate}

\end{Proposition} 
\begin{proof}
         (1) Recall the exact sequence \eqref{split as three barbell}:
         \[
            \pi_1^D(\Emb_{\partial}(I\sqcup I,X_1))\rightarrow\pi_1(\Emb_{\partial}(I\sqcup I,X_1))\xrightarrow{\mathcal{F}}\pi_2(X_1)\oplus \pi_2(X_1)
         \]
         By \cite[Lemma 3.15]{LXZ25Dax}, the first component of $\pi_2(X_1)\oplus \pi_2(X_1)$ is generated by $\{\mathcal{F}(\rho_{\gamma})\mid \gamma\in\pi_1(X_1;I_1,S^2\times\{q\})\}\cup\{\mathcal{F}(\eta_{\lambda})\mid\lambda\in\pi_1(X_1;I_1,G)\}$. The second component has a similar set of generators, replacing $I_1$ by $I_2$. Since $\pi_1^D(\Emb_{\partial}(I\sqcup I,X_1))$ is generated by $\{\tau_g|g\in\pi_1(X_1;I_i,I_j)\}$, we know that $\pi_1(\Emb_{\partial}(I\sqcup I, X_1))$ is generated by $\tau_g,\rho_{\gamma},\eta_{\lambda}$. The commutativity of $\pi_1(\Emb_{\partial}(I\sqcup I, X_1))$ follows from \cite[Lemma 3.25]{LXZ25Dax}.
         
         (2) is \cite[Lemma 3.26]{LXZ25Dax}.
     \end{proof}

\subsection{Computing the action of $\pi_{1}(\Emb'_{\partial }(I\sqcup I,X_1))$}

Now we study the action of $\pi_{1}(\Emb'_{\partial }(I\sqcup I,X_1))\cong \pi_{1}(\Emb_{\partial}(H_0\cup H_1,X'))$ on $\pi_2(X^{\prime},\partial D^2_0)$ via the map $e$ defined in (\ref{eq: extension}). It suffices to consider the action of the generators provided by Proposition \ref{description of normal section}.

\begin{Proposition} For any $\sigma\in \pi_{1}(X')$, we have 
\[
[e(\xi_{1})(D^2_{\sigma})]-[D^2_{\sigma}]=((1,0)+(-1,0)-2(0,0),0,0)\in \widetilde{\mathcal{J}}.
\]
and 
  \[
[e(\xi_{2})(D^2_{\sigma})]-[D^2_{\sigma}]=((0,1)+(0,-1)-2(0,0),0,0)\in \widetilde{\mathcal{J}}.
\]  
\end{Proposition}
\begin{proof}
    By the construction in \cite[Lemma 4.10]{LXZ25Dax}, $e(\xi_1)(D^2_{\sigma})$ is homotopic to $D^2_{\sigma}\#_{\gamma} m_{1} \#_{\gamma'} \bar{m}'_{1}$. Here $m'_{1}$ is a parallel copy of $m_{1}$ and $\bar{m}'_{1}$ denotes its orientation reversal. The path $\gamma$ is  contained in a small disk around $x^{+}_{1}$. And  the path $\gamma'$ is contained in a small disk around $x^{-}_{1}$. By Lemma \ref{lem: tubing}, we have
    \[
[D^2_{\sigma}\#_{\gamma} m_{1} \#_{\gamma'} \bar{m}'_{1}]-[D^2_{\sigma}\#_{\gamma} m_{1} ]=((0,0),0,0)-((-1,0),0,0)
    \]
and
  \[
[D^2_{\sigma}\#_{\gamma} m_{1}]-[D^2_{\sigma}]=((1,0),0,0)-((0,0),0,0).
\]
This proves the first case. The second case is similar.
    \end{proof}

Next, we study the effect of the barbell diffeomorphisms homotopy class of $D^2_{\sigma}$. First we consider barbell diffeomorphism extending the spinning families $\eta_{\lambda}$, where $\lambda\in\pi_1(X^{\prime};I_u,G)$. Recall that $m_u$ intersects $D^2_{\sigma}$ at $x^+_{u}$ and $x^-_{u}$. And $G$ intersects $\Sigma_{\sigma}$ at a single point $x_G$. We represent $\lambda$ by a path  $s:I\rightarrow X^{\prime}$ from $I_u$ to $G$, with $\operatorname{int} s \cap (m_u\cup G)=\emptyset$. Take a regular neighborhood of $m_u\cup s\cup G$, which is the barbell
    \[
        \mathcal{B}_{\lambda}=\nu(m_u\cup s\cup G)\cong (S^2\times D^2)\natural (S^2\times D^2)
    \]
Homotope the endpoints of $s$ along $m_k\cup G$ to obtain a path $s_+$ from $x_{u}^+$ to $x_G$ and a path $s_-$ from $x^{-}_{u}$ to $x_G$. Thus, $s_+, s_-$ determine elements of $\pi_1(X^{\prime};D^2_{\sigma}\cup G)$. Because $D^2_{\sigma}\cup G$ is simply connected, $s_+, s_-$ can be viewed as elements in $\pi_1(X^{\prime})$. Denote the corresponding homotopy classes by $\lambda_+$ and $\lambda_-$, respectively. 

\begin{Lemma}\label{calcu of meridian to G}
    For any $\lambda\in\pi_1(X^{\prime};I_u,G)$, $\sigma\in\pi_1(\Sigma_0)$, we have 
    \[e(\eta_{\lambda})([D^2_{\sigma}])-[D^2_{\sigma}]
=        (0,0,\lambda_+-\lambda_{-})-(\lambda_+^{-1}-\lambda_-^{-1},0,0)
    \in \widetilde{\mathcal{J}}.\]
\end{Lemma}
\begin{proof}
    The intersection between $\mathcal{B}_{\lambda}$ and $D^2_{\sigma}$ is a union of  three disks $D_+, D_-, D_G$,  with centers $x_{+}^{u}, x^{u}_{-}$ and $x_{G}$ respectively. Since $e(\eta_{\lambda})$ is obtained by extending the barbell diffeomorphism on $\mathcal{B}_{\lambda}$ with the identity, $e(\eta_{\lambda})(D^2_{\sigma})$ is obtained from $D^2_{\sigma}$ by replacing $D_+,D_-,D_G$ with $e(\eta_{\lambda})(D_+),e(\eta_{\lambda})(D_-),e(\eta_{\lambda})(D_G)$ respectively. Moreover, according to the construction of barbell diffeomorphisms, $e(\eta_{\lambda})(D_+)$ is obtained by tubing $D_+$ with a parallel copy $G'$ of $G$ along $s_+$,  $e(\eta_{\lambda})(D_-)$ is obtained by tubing $D_-$ with a parallel copy $\overline{G}''$ of $\overline{G}$ along $s_-$, and $e(\eta_{\lambda})(D_G)$ is obtained by tubing $D_G$ with $\overline{m_u}$ along $s^{-1}$. In other words, we have 
    \[
    [e(\eta_{\lambda})(D^2_{\sigma})]=[D^2_{\sigma}\#_{s_+} G'\#_{s_{-}} \bar{G}''
    \#_{s^{-1}}m_{u}]
    \]
    Then we apply Lemma \ref{lem: tubing} to get 
    \[
    [e(\eta_{\lambda})(D^2_{\sigma})]-[D^2_{\sigma}]=(0,0,\lambda_+-\lambda_{-})-(\lambda_+^{-1}-\lambda_-^{-1},0,0).
    \]
\end{proof}

\begin{Lemma}\label{calculate for meridian and S^2*q} 
    Suppose $\gamma\in\pi_1(X_1;I_u,R)$and $\sigma\in\pi_1(\Sigma_0)$. Then the difference $e(\rho_{\gamma})([D^2_{\sigma}])-[D^2_{\sigma}]\in \pi_2(X^{\prime},\partial D^2_0)$ equals to
    \[
        ((\sigma-1)\cdot(\gamma_+^{-1}-\gamma_-^{-1}), \gamma_+-\gamma_-, 0)
    \]
\end{Lemma}
\begin{proof}
    Recall that $\Sigma_{\sigma}=\Sigma_0\#_{\kappa}(G_1\#_{\sigma}\overline{G_2})$, where $G_{1}, G_{2}$ are two parallel copies of $G$. Hence the meridian $m_k$ intersects 
    $D^2_{\sigma}$ at two points, by $\{x^+_u,x^-_u\}$, with different signs. And $R$ also intersects $D^2_{\sigma}$ at two points, denoted by $y_+\in G_1, y_-\in \overline{G_2}$, with opposite signs. Therefore, the barbell $\mathcal{B}_{\gamma}=\nu(m_u)\natural \nu(R)$ intersects $D^2_{\sigma}$ at four small disks $D_{x+}, D_{x-}, D_{y+}, D_{y-}$, centered at $x^+_u,x^-_u,y_+,y_-$ respectively. The disk $e(\rho_{\gamma})(D_{x+})$ is obtained by tubing $D_{x+}$ with $R$ along $\gamma_+$. The disk $e(\rho_{\gamma})(D_{x-})$ is obtained by tubing $D_{x-}$ with $\overline{R}$ along $\gamma_-$. The disk $e(\rho_{\gamma})(D_{y_+})$ is obtained by tubing $D_{y_+}$ with $\overline{m_k}$ along $\gamma^{-1}$. The disk $e(\rho_{\gamma})(D_{y_-})$ is obtained by tubing $D_{y_-}$ with $m_k$ along $\sigma\cdot\gamma^{-1}$ since $y_-\in G_2$ and $D^2_{\sigma}=D^2_0\#_{\kappa}(G_1\#_{\sigma}\overline{G_2})$. The rest of proof is the same as Lemma \ref{calcu of meridian to G}.
\end{proof}

As the last case, we consider the spinning families $\tau_g$ with $g\in\pi_1(X_1;I_u,I_v)$. We represent $g$ by a path $\alpha:[0,1]\to X_1$ from $I_u$ to $I_{v}$. By perturbing $\alpha|_{[0,\epsilon)}$ in the positive direction and perturbing $\alpha|_{[1-\epsilon,1]}$ in the negative direction, we obtain a path $\alpha_{+-}:[0,1]\to X'$ from $x^{+}_{u}\in D^2_0$ to $x^{-}_{v}\in D^2_0$, which represents an element $g_{+-}\in \pi_{1}(X',D^2_0)\cong \pi_{1}(X')$. We define $\alpha_{-+}, \alpha_{--}, \alpha_{++}$ similarly. They represent elements $g_{-+}, g_{--}, g_{++}\in \pi_{1}(X',D^2_0)\cong \pi_{1}(X')$. 

\begin{Lemma}\label{lemma for mk and ml}
   The difference $e(\tau_{g})([D^2_{\sigma}])-[D^2_{\sigma}]\in \widetilde{\mathcal{J}}$ equals to $(k_g,0,0)$, where $k_g\in\mathcal{J}$ is given by
     \[
((g_{++}-g_{+-})-(g_{-+}-g_{--}))-((g^{-1}_{++}-g^{-1}_{+-})-(g^{-1}_{-+}-g^{-}_{--}))
 \]  
 \end{Lemma}
\begin{proof}
We perturb $\beta$ to a path $\beta'$ from $m_{u}$ to $m_{v}$. And we consider the barbell $\mathcal{B}_{g}=\nu(m_{u}\cup \beta'\cup m_{v})$. Then 
$e(\tau_{g})(D^2_{\sigma})$ is isotopic the image of $D^2_{\sigma}$ under the barbell diffeomorphism implemented along $\mathcal{B}_{g}$. Note that  $\mathcal{B}_{g}\cap D^2_{\alpha}$ is a disjoint union of 4-disks small $D_{++}, D_{+-}, D_{-+}, D_{++}$. Therefore, we have 
\[
[e(\tau_{g})(D^2_{\sigma})]=[D^2_{\sigma}\#_{\beta_{+,*}} (m_{v}) \#_{\beta_{-,*}} (\overline{m_{v}})\#_{\beta_{*,+}} (\overline{m_{u}})\#_{\beta_{*,-}}(m_{u})]
\]
Here $\beta_{*,+}:I\to X'$ is the path from $m_{u}$ to $x^{+}_{v}$ obtained by perturbing $\beta$. Then we apply Lemma \ref{lem: tubing} four times.\end{proof}
\begin{Corollary}
We have $e(\tau_{g})([D^2_{\sigma}])-[D^2_{\sigma}]\in \mathcal{J}^2\oplus \mathbb{Z}[\pi_{1}(X')]^{\oplus 2}  $ \end{Corollary}
\begin{proof}
Let $g_{1}=(1,0)\in \pi_{1}(X')$ and $g_{2}=(0,1)\in \pi_{1}(X')$. Then 
\[
g_{++}-g_{+-}-g_{-+}+g_{--}=g_{--}\cdot (g_{1}g_{2}-g_{1}-g_{2}+1)=g_{--}\cdot (g_{1}-1)\cdot (g_{2}-1)\in \mathcal{J}^2.
\]
Similarly, $g^{-1}_{++}-g^{-1}_{+-}-g^{-1}_{-+}+g^{-1}_{--}\in \mathcal{J}^2$.
\end{proof}


\section{Proof of Theorem \ref{main thm}}
\label{sec_proof}
First we need some algebraic preparations. 
Note that $\mathcal{J}$ is a free abelian group with basis $\{\sigma-1|\sigma\neq1\in\pi_1(X')\}$ and $\mathcal{J}^2$ is an abelian group generated by $\{(\sigma_1-1)(\sigma_2-1)|\sigma_1,\sigma_2\in\pi_1(X')\}$. To distinguish from the addition in $\mathbb{Z}[\pi_1(X')]$, we write the addition in $\pi_1(X')$ as multiplication. Since $\pi_1(X')\cong\mathbb{Z}^{4}$, we know that $\mathcal{J}^2$ is generated by  
\[
    \{(a_1+b_1,\cdots,a_4+b_4)-(a_1,\cdots,a_4)-(b_1,\cdots,b_4)+(0,\cdots,0)|(a_1,\cdots,a_4),(b_1,\cdots,b_4)\in\pi_1(X')\}
\]
Therefore, for any $(a_1,\cdots,a_4),(b_1,\cdots,b_4)\in\pi_1(X')$, we have the equation in $\mathbb{Z}[\pi_1(X')]/\mathcal{J}^2$:
\begin{equation}\label{subtraction in group algebra}
    (a_1,\cdots,a_4)-(b_1,\cdots,b_4)=(a_1-b_1,\cdots,a_4-b_4)-(0,\cdots,0).
\end{equation}
It is not hard to deduce from \eqref{subtraction in group algebra} the following equations in $\mathbb{Z}[\pi_1(X')]/\mathcal{J}^2$:
\[
    (a_1,\cdots,a_4)-(0,0)=(0,0)-(-a_1,\cdots,-a_4),
\]
\[
    k(a_1,\cdots,a_4)-k(0,\cdots,0)=(ka_1,\cdots,ka_4)-(0,0),
\]
\[
     (a_1,\cdots,a_4)-(0,0)+(b_1,\cdots,b_4)-(0,0)=(a_1+b_1,\cdots,a_4+b_4)-(0,0),
\]
where $(a_1,\cdots,a_4),(b_1,\cdots,b_4)\in\pi_1(X')$ and $k\in\mathbb{Z}$. Consider the homomorphism between abelian groups $F:\mathcal{J}\rightarrow \mathbb{Z}^4,$ by extending \[(a_1,\cdots,a_4)-(0,\cdots,0)\mapsto  (a_1,\cdots,a_4)\] $\mathbb{Z}$-linearly.
\begin{Lemma}\label{isomorphism for J/J^2}
    The map $F$ induces an isomorphism $\overline{F}:\mathcal{J}/\mathcal{J}^2\rightarrow \mathbb{Z}^4$
\end{Lemma}
\begin{proof}
    Since $\mathcal{J}^2$ is generated by elements of the form:
    \[
        (a_1+b_1,\cdots,a_4+b_4)-(a_1,\cdots,a_4)-(b_1,\cdots,b_4)+(0,\cdots,0),
    \]
    we know that $\mathcal{J}^2\subset \operatorname{ker}F$ and $F$ can be reduced to $\overline{F}:\mathcal{J}/\mathcal{J}^2\to \mathbb{Z}^4$. Obviously, $\overline{F}$ is surjective. For any $\sum\limits_{i=1}^{n}r_i((a_{i1},\cdots,a_{i4})-(0,\cdots,0))\in \mathcal{J}/\mathcal{J}^2$, where $\sum\limits_{i=1}^n r_i=0$, it can be deduced from the three equations above that 
    \[
        \sum\limits_{i=1}^{n}r_i((a_{i1},\cdots,a_{i4})-(0,\cdots,0))=\left(\sum\limits_{i=1}^{n-1}r_i(a_{i1}-a_{n1}),\cdots, \sum\limits_{i=1}^{n-1}r_i(a_{i4}-a_{n4})\right)-(0,\cdots,0).
    \]
    Hence, any element in $\mathcal{J}/\mathcal{J}^2$ can be written as $(a_1,\cdots,a_4)-(0,\cdots,0)$ for some $(a_1,\cdots,a_4)\in\pi_1(X')$. For any $x\in\operatorname{ker}\overline{F}$, write $x$ as $(a_1,\cdots,a_4)-(0,\cdots,0)$. Then
    \[
        0=\overline{F}(x)=\overline{F}\left((a_1,\cdots,a_4)-(0,\cdots,0)\right)=(a_1,\cdots,a_4),
    \]
    and we know that $x=0\in \mathcal{J}/\mathcal{J}^2$. Thus, $\overline{F}$ is injective.
\end{proof}   

Now we can proof Theorem \ref{main thm}:
\begin{proof}[Proof of Theorem \ref{main thm}]
    (1) The fact that $\{i_{\sigma}\}$ are homotopic follows from \cite[Lemma 4.2]{LXZ26pseudoisotopies3manifoldsinfinitefundamental} directly. Recall from Section \ref{sec_constru} that for any $\sigma\in\pi_1(\Sigma_0)$, we construct an embedded surface $\Sigma_{\sigma}=\Sigma_0\#_{\kappa}G_{12}$, where $G_{12}=G_1\#_{\sigma}\overline{G_2}$ and $\kappa:I\rightarrow X$ is a trivial path from $\Sigma_0$ to $G_1$ with $\operatorname{int}\kappa\cap(\Sigma_0\cup G_1\cup G_2)=\emptyset$. Pick another copy $G_3=S^2\times\{p_3\}$ of the geometric dual $G$ which is very close to $G_1=S^2\times\{p_1\}$ and homotope one endpoint of $\kappa$ along $\Sigma_0\cup G_3$ to get a trivial path $\kappa_1$ from $G_3$ to  $G_1$ with $\operatorname{int}\kappa_1\cap(\Sigma_0\cup G_3\cup G_{12})=\emptyset$. Note that the  barbell diffeomorphism $f_{\sigma}$ on $X$ corresponding to the barbell $\mathcal{B}=\nu(G_3\cup\kappa_1\cup G_{12})$ maps $\Sigma_0$ to $\Sigma_{\sigma}$. Hence, for any $\sigma_1,\sigma_2\in\pi_1(\Sigma_0)$, $\Sigma_{\sigma_1}$ is mapped to $\Sigma_{\sigma_2}$ by $f_{\sigma_2}\circ f_{\sigma_1}^{-1}$. As a consequence, all surfaces $\Sigma_{\sigma}$ have diffeomorphic complements in $X$.
    
    (2) Suppose $\sigma_1\neq\sigma_2\in \pi_1(\Sigma_0)$. 
    Pick any isotopy $\mathcal{I}$ from $i_{\sigma_1}$ to $i_1'$ with $i_{\sigma_1}'|_{H_0\cup H_1}=i_{\sigma_2}|_{H_0\cup H_1}$, and restrict $\mathcal{I}$ to $H_0\cup H_1$. Then we get a loop in $\Emb(H_0\cup H_1,X)$, denoted by $\theta$. Note that $i_{\sigma_1}'|_{(i_{\sigma_1}')^{-1}(X^{\prime})}$ equals to $e(\theta)$ in $\pi_0(\Diff(X^{\prime},i_0(H_2)\cap X^{\prime}))$. Therefore the obstruction $o(i_{\sigma_1},i_{\sigma_2},\mathcal{I})$ is exactly \[[e(\theta)(D^2_{\sigma_1})]-[D^2_{\sigma_2}]\in \widetilde{\mathcal{J}}.\] Apply $e(\theta^{-1})$ to both side of the following equation:
    \begin{equation*}
        [e(\theta)(D^2_{\sigma_1})]-[D^2_{\sigma_2}]=e(\theta)([D^2_{\sigma_1}]-[D^2_{\sigma_2}])+([e(\theta)(D^2_{\sigma_2})]-[D^2_{\sigma_2}])
    \end{equation*}
    it suffices to prove \[[D^2_{\sigma_1}]-[D^2_{\sigma_2}]-([e(\theta^{-1})(D^2_{\sigma_2})]-[D^2_{\sigma_2}])\neq 0.\] For this, we first calculate $[e(\theta^{-1})(D^2_{\sigma_2})]-[D^2_{\sigma_2}]$. Since $\theta$ is an arbitrary element in $\pi_1(\Emb(H_0\cup H_1,X))$, we may replace $\theta$ by $\theta^{-1}$ and compute $[e(\theta)(D^2_{\sigma_2})]-[D^2_{\sigma_2}]$ instead. By Proposition \ref{move of basepoint}, there exists another \[\theta^{\prime}\in\pi_1(\Emb(H_0\cup H_1,X)),\] such that $\mathrm{Res}(\theta^{\prime})=\mathrm{Res}(\theta)$, and \[e(\theta^{\prime})([D^2_{\sigma_2}])-[D^2_{\sigma_2}]=(0,1-\overline{\theta},0),\] where $\overline{\theta}=\mathrm{Res}(\theta)\in\pi_1(X,x_0)\cong\pi_1(X^{\prime})$. Therefore, $(\theta^{\prime})^{-1}\theta\in \ker \mathrm{Res}$. Since we have the exact sequence \eqref{split as basepoint}, $(\theta^{\prime})^{-1}\theta$ lies in the image of $\pi_1(\Emb^{\prime}_{\partial}(I\sqcup I,X_1))\rightarrow \pi_1(\Emb(H_0\cup H_1,X))$. 
    
    By Proposition \ref{description of normal section}, $\pi_1(\Emb^{\prime}_{\partial}(I\sqcup I,X_1))$ is generated by $\hat{\tau}_g, \hat{\eta}_{\lambda}, \hat{\rho}_{\gamma}$ and $\xi_1,\xi_2$. Hence, we have 
    \[(\theta^{\prime})^{-1}\theta=\xi_1^{t_1}\circ \xi_2^{t_2}\circ\prod_{i}\rho_{\gamma_i}\circ\prod_{j}\tau_{g_j}\circ\prod_{r}\eta_{\lambda_r}\] for some $\gamma_i\in\pi_1(X^{\prime};I_u,R), g_j\in\pi_1(X^{\prime};I_u,I_v), \lambda_r\in \pi_1(X^{\prime};I_u,G), t_1,t_2\in\mathbb{Z}$.  By Lemma \ref{calcu of meridian to G}, Lemma \ref{calculate for meridian and S^2*q} and Lemma \ref{lemma for mk and ml},
    \begin{equation*}
        \begin{split}
            [e(\theta)(D^2_{\sigma_2})]-[D^2_{\sigma_2}]=&(0,1-\overline{\theta},0)\\
            &+t_1\cdot (((1,0)+(-1,0)-2(0,0)),0,0)\\
            &+t_2\cdot (((0,1)+(0,-1)-2(0,0)),0,0)\\
            &+\sum\limits_i((\sigma-1)\cdot(\gamma_{i+}^{-1}-\gamma_{i-}^{-1}), \gamma_{i+}-\gamma_{i-}, 0)\\
            &+(\sum\limits_{j}k_{g_j},0,0)
            \\&+(0,0,\sum\limits_{r} \lambda_{r+}-\lambda_{r-})-(\sum\limits_{r} \lambda_{r+}^{-1}-\lambda_{r-}^{-1},0,0),
        \end{split}
    \end{equation*}
    where $k_{g_j}\in \mathcal{J}^2$.
    
   Note that $[D^2_{\sigma_1}]-[D^2_{\sigma_2}]=(0,0,-\sigma_1+\sigma_2)$. If there exists an isotopy $\mathcal{I}$ such that $o(i_{\sigma_1},i_{\sigma_2},\mathcal{I})=0$, we must have  $\sum\limits_r(\lambda_{r+}-\lambda_{r-})=-\sigma_1+\sigma_2\in \mathbb{Z}[\pi_1(X^{\prime})]$ by considering the third component. Hence $\sum\limits_r(\lambda_{r+}^{-1}-\lambda_{r-}^{-1})=\sigma_1^{-1}-\sigma_2^{-1}$. On the other hand, we have the following equation by considering the first component:
    \begin{equation}\label{final calculation}
        \begin{split}
             \sum\limits_r(\lambda_{r+}^{-1}-\lambda_{r-}^{-1})&=t_1(((1,0)+(-1,0)-2(0,0))\\&+t_2(((0,1)+(0,-1)-2(0,0))\\
             &+\sum\limits_{j}k_{g_j}+\sum\limits_i(\sigma-1)\cdot(\gamma_{i+}^{-1}-\gamma_{i-}^{-1})
        \end{split}
    \end{equation}
     Therefore, the right hand side of \eqref{final calculation} belongs to the ideal $\mathcal{J}^2$. However, the left hand side of \eqref{final calculation} equals to $\sigma_1^{-1}-\sigma_2^{-1}$. As a consequence, $\sigma_1^{-1}-\sigma_2^{-1}=0\in \mathcal{J}/\mathcal{J}^2$,
    which is impossible by the isomorphism $\overline{F}$ constructed in Lemma \ref{isomorphism for J/J^2}.

    (3)
    Denote the result of finitely many times of external stabilizations of $X$ by $X_{[n]}=X\#n(S^2\times S^2)$. Since the connected sum is perfomed away from the surfaces $\Sigma_{\sigma}$, $i_{\sigma}$ can be viewed as an embedding into $X_{[n]}$. We denote the closure of $X_{[n]}-\nu(i_0(H_0\cup H_1))$ in $X_n$ by $X_{[n]}^{\prime}$. Then 
    \[
        \pi_2(X_{[n]}^{\prime},\partial D^2_0)\cong H_2(\widetilde{X_{[n]}^{\prime}})\cong H_2(\widetilde{X^{\prime}})\oplus A,
    \] where $A=\oplus_{n}(\mathbb{Z}[\pi_1(X^{\prime})]\oplus \mathbb{Z}[\pi_1(X^{\prime})])$, each summand in which is generated by the $S^2\times S^2$ in the external stabilization part in $X_n$. Denote the closure of $X_{[n]}-\nu(i_0(H_0))$ in $X_{[n]}$ by $(X_{[n]})_1$. Then $\pi_1(\Emb^{\prime}_{\partial}(I\sqcup I,(X_{[n]})_1))$ is generated by $\hat{\tau}_g, \hat{\eta}_{\lambda}, \hat{\rho}_{\gamma}, \hat{\delta}_{\mu}$ and $\xi_1,\xi_2$, where $\hat{\tau}_g, \hat{\eta}_{\lambda}, \hat{\rho}_{\gamma}$ and $\xi_1,\xi_2$ are the same as before and $\hat{\delta}_{\mu}$ is the spinning family of $I_k$ around some copy $S$ along $\mu\in\pi_1(X_{[n]}^{\prime};I_k,S)$. Here, $S$ stands for $S^2\times\{*\}$ or $\{*\}\times S^2$ in the external stabilization part in $X_{[n]}$. 

    Similar to (2), for any isotopy $\mathcal{I}$ in $X_{[n]}$ from $i_1$ to $i_1'$ with $i_1'|_{H_0\cup H_1}=i_2|_{H_0\cup H_1}$, we restrict it to $H_0\cup H_1$ to get a loop in $\Emb(H_0\cup H_1,X_{[n]})$. As in (2), 
    \[
        \theta=(\theta^{\prime})^{-1}\circ\xi_1^{t_1}\circ \xi_2^{t_2}\circ\prod_{i}\rho_{\gamma_i}\circ\prod_{j}\tau_{g_j}\circ\prod_{r}\eta_{\lambda_r}\circ\prod_{t}\delta_{\mu_t},
    \]
     for some $\gamma_i\in\pi_1(X^{\prime};I_k,S^2\times\{q\}),$ $ g_j\in\pi_1(X^{\prime};I_k,I_l), $$\lambda_r\in \pi_1(X^{\prime};I_k,G),\mu_t\in\pi_1(X_{[n]}^{\prime};I_k,S)), t_1,t_2\in\mathbb{Z}$.
    Note that the barbell $\mathcal{B}_{\mu_t}=\nu(m_k)\natural \nu(S)$ intersects $D^2_{\sigma_2}$ in $\{x_+,x_-\}\subset m_k$ transversely. $e(\delta_{\mu_t})(D^2_{\sigma_2})$ is obtained by tubing $D^2_{\sigma_2}$ with  $S$ , $\overline{S}$ near $x_+$ , $x_-$, respectively. Therefore, 
    \[
        e(\delta_{\mu_t})([D^2_{\sigma_2}])-[D^2_{\sigma_2}]=(0,\mu_{t+}-\mu_{t-})\in H_2(\widetilde{X^{\prime}})\oplus A
    \]
    Hence, the argument in (2) still works by considering the $H_2(\widetilde{X^{\prime}})$ component of $[e(\theta)(D^2_{\sigma_2})]-[D^2_{\sigma_2}]$.
\end{proof}

\bibliographystyle{amsalpha}
\bibliography{references}

\end{document}